\documentclass[10pt]{amsart}

\title[Equivariant LG mirror symmetry]{Equivariant Landau--Ginzburg mirror symmetry}

\author[Gu\'er\'e]{J\'er\'emy Gu\'er\'e}
\address{Univ. Grenoble Alpes, CNRS, IF, 38000 Grenoble, France}
\email{jeremy.guere@gmail.com}

\usepackage{amsrefs}
\usepackage{amsfonts}
\usepackage{amsmath}
\usepackage{amssymb}
\usepackage{amscd}
\usepackage{array}
\usepackage{subfigure}
\usepackage{tikz}
\usetikzlibrary{arrows}

\usepackage[all]{xy}

\allowdisplaybreaks[3]

\newcommand{\vir}{\mathrm{vir}}
\newcommand{\rk}{\mathrm{rk}}
\newcommand{\KK}{\mathbb{K}}

\newcommand{\grj}{\mathfrak{j}}

\newcommand{\fc}{\mathfrak{c}}

\newcommand{\rR}{\mathfrak{C}}

\newcommand{\CC}{\mathbb{C}}
\newcommand{\ZZ}{\mathbb{Z}}

\newcommand{\QQ}{\mathbb{Q}}

\newcommand{\cO}{\mathcal{O}}

\renewcommand{\(}{\left(}
\renewcommand{\)}{\right)}

\newcommand{\cC}{\mathcal{C}}
\newcommand{\cD}{\mathcal{D}}

\newcommand{\cL}{\mathcal{L}}
\newcommand{\cM}{\mathcal{M}}

\newcommand{\fq}{\mathfrak{q}}

\newcommand{\cS}{\mathrm{Sym}}
\newcommand{\sS}{\mathcal{S}}

\newcommand{\Ch}{\mathrm{Ch}}
\newcommand{\Td}{\mathrm{Td}}

\newcommand{\st}{\mathbf{H}}

\theoremstyle{plain}

\newtheorem{thm}{Theorem}[section]
\newtheorem{pro}[thm]{Proposition}

\newtheorem*{lem*}{Lemma}

\newtheorem{cor}[thm]{Corollary}

\theoremstyle{definition}

\newtheorem{dfn}[thm]{Definition}

\newtheorem{rem}[thm]{Remark}
\newtheorem*{rem*}{Remark}
\newtheorem*{rems*}{Remarks}

\newtheorem*{exa*}{Example}
\newtheorem*{exait*}{\rm \em Example}
\newtheorem*{exadefit*}{\rm \em Example/Definition}
\newtheorem*{cla*}{\rm \em Claim}

\newtheorem*{dfn*}{Definition}

\usepackage{amsxtra}

\def\<{\left\langle}
\def\>{\right\rangle}

\begin{document}
\begin{abstract}
We give a new proof of the computation of Hodge integrals we have previously obtained for the quantum singularity (FJRW) theory of chain polynomials.
It uses the classical localization formula of Atiyah--Bott and we phrase our proof in a general framework that is suitable for future studies of gauged linear sigma models (GLSM).
As a by-product, we obtain the first equivariant version of mirror symmetry without concavity, generalizing the work of Chiodo--Iritani--Ruan on the Landau--Ginzburg side.
\end{abstract}

\maketitle

\tableofcontents

\setcounter{section}{-1}

\section{Introduction}
In our previous work \cite{Guere1}, we gave the first genus-zero computation of the virtual cycle in the quantum singularity theory, also called Fan--Jarvis--Ruan--Witten (FJRW) theory \cite{FJRW1,FJRW2,Polish1}, in a range of cases where the state-of-the-art techniques relying on the concavity condition did not apply. As an application, we proved a mirror symmetry theorem for these theories.

Later in \cite{Guere2}, we generalized our results and obtained the first all-genus computation on the moduli space of Landau--Ginzburg spin curves, providing we first cap the virtual cycle with the Euler class of the Hodge vector bundle. It lead to Hodge integral calculations in the quantum singularity (FJRW) theory in a range of cases where the techniques relying on Teleman's reconstruction theorem for generically semi-simple Cohomological Field Theories (CohFTs) did not apply. As an application, we proved in \cite{Guere3} the DR/DZ conjecture for $3$-spin, $4$-spin, and $5$-spin theories, that is the equivalence of the Double Ramification (DR) hierarchy with the Dubrovin--Zhang (DZ) hierarchy for these theories.

Interestingly, there are up to now no counterparts of \cite{Guere1,Guere2} for Gromov--Witten theory of hypersurfaces in weighted projective spaces, although such a parallel story should appear under the light of Landau--Ginzburg/Calabi--Yau correspondence \cite{LGCY}.
Genus-zero Gromov--Witten invariants of hypersurfaces in weighted projective spaces are still unknown, as soon as the convexity condition fails. Even for projective hypersurfaces, there is no general description of Hodge integrals in positive genus.

Both papers \cite{Guere1,Guere2} relied on a new technique based on the notion of recursive complexes and it has been our main focus for the last five years to understand how to carry this technique into the Gromov--Witten side.
To achieve this ambitious goal, we shed new light on our previous results by changing our strategy to a more Gromov--Witten-like approach: we make use of the localization formula of Atiyah--Bott \cite{Atiyah}, developped in the algebraic category by \cite{Edidin,Edidin2}, to carry the computation of Hodge integrals in FJRW theory.
We then give a new and shorter proof of the results in \cite{Guere1,Guere2}.
We also upgrade our mirror theorem \cite[Theorem 4.4]{Guere1} to an equivariant version of it, in the spirit of \cite[Section 4.3]{LGCY}, see Theorem \ref{equiMS}.
We highlight the fact it was out of reach with the previous technique.

Importantly, we phrase our new method in a very general framework that is relevant when working with Landau--Ginzburg models. We thus believe it is suitable for the study of any Gauged Linear Sigma Model (GLSM) \cite{GLSM}.
Indeed, following our work \cite[Section 6]{Guere8}, we see that the definition of virtual cycles in hybrid GLSMs can be phrased as a localized Chern character of a two-periodic complex on a big moduli space denoted $\square$ in \cite{Guere8}, and that the picture in \cite[Section 1.5]{Guere8} is a special case of the one we describe in Section \ref{section1}.
Furthermore, it is worth noticing that Gromov--Witten theory of a complete intersection in a toric Deligne--Mumford (DM) stack is a special instance of a GLSM via the comparaison \cite{Jeongseok}.
In particular, it is absolutely clear that the strategy developed in Section \ref{section1} applies, with little changes compared to Section \ref{section2}, to the case of hypersurfaces in weighted projective spaces which are defined by chain polynomials.
However, writing this paper, we discovered a more direct way to pursue this goal and we decided to leave this result to another paper \cite{Guere10}.



\noindent
\textbf{Acknowledgement.}
The author is grateful to Alexander Polishchuk who suggested first to look for a proof of the results in \cite{Guere1,Guere2} using the localization formula.

\section{Localization formula for localized Chern characters}\label{section1}
Here, we describe the localization method that we apply to FJRW theory in the next section.
We explain it in a more general framework, so that it can serve as a reference for future works regarding GLSM models.

\subsection{Localized Chern character}\label{section1.1}
We work over an arbitrary field $\KK$ and we consider the following set-up:
\begin{center}
\begin{tikzpicture}[scale=1.5]
\node (A) at (-1,1) {$Y$};
\node (B) at (0,1) {$X$};
\node (C) at (0,0) {$S$};
\node (D) at (0.5,0.5) {$E$};
\node (E) at (0.5,1.5) {$V$};
\node (F) at (1.25,0.75) {$\circlearrowleft \KK^*$};
\node (G) at (-1.5,1.5) {$T$};
\draw[->] (B) -- (C);
\draw[right hook->] (A) -- (B);
\draw[->] (D) to[bend left=20] (C);
\draw[->] (E) to[bend left=20] (B);
\draw[->] (G) to[bend right=20] (A);
\node[left] at (0,0.5){$p$};
\node[above] at (-0.5,1){$j$};
\end{tikzpicture}
\end{center}
where $S$ is a proper DM stack, $X$ is a DM stack over $S$, the substack $Y$ is a local complete intersection in $X$ and is proper over $S$, and $V, E$, and $T$ are locally free sheaves (vector bundles) over the DM stacks $X, S$, and $Y$.
Moreover, we have an action of the multiplicative group $\KK^*$ on the fibers of $p$; precisely an action on $X$ and on $S$ such that the action on $S$ is trivial and the projection morphism $p \colon X \to S$ is equivariant.
We assume the closed substack $Y$ to be $\KK^*$-invariant and the vector bundles $V, E$, and $T$ to be $\KK^*$-equivariant.
Furthermore, we assume the $\KK^*$-fixed locus of $X$ to be a closed substack of $Y$ that we denote $Y_F$.

We also consider four global sections
$$\alpha, \alpha' \in H^0(X,V^\vee) ~ , ~~ \alpha'' \in H^0(X,p^*E) ~, ~~ \textrm{and} ~~ \beta \in H^0(X,V)$$
such that $\alpha, \alpha''$, and $\beta$ are $\KK^*$-equivariant but $\alpha'$ is not, and that
$$\beta(\alpha) = \beta(\alpha') = 0.$$
For every $t_1,t_2 \in \KK$, we define global sections
$$\alpha(t_1) = \alpha + t_1 \alpha' \in H^0(X,V^\vee) \quad \textrm{and} \quad \widetilde{\alpha}(t_1,t_2) = \alpha + t_1 \alpha' + t_2 \alpha'' \in H^0(X,V^\vee \oplus p^*E)$$
and Koszul two-periodic complexes
$$K(t_1) = \left(\Lambda^\bullet (V^\vee), \delta(t_1) \right) \quad \textrm{and} \quad \widetilde{K}(t_1,t_2) = \left(\Lambda^\bullet (V^\vee \oplus p^* E), \widetilde{\delta}(t_1, t_2) \right)$$
over the DM stack $X$, where the maps are
$$\delta(t_1) = \alpha(t_1) \wedge \cdot + \beta(\cdot) \quad \textrm{and} \quad \widetilde{\delta}(t_1, t_2) =\widetilde{\alpha}(t_1,t_2) \wedge \cdot + \beta(\cdot).$$
We observe the following
\begin{eqnarray*}
K(t_1) \textrm{ is $\KK^*$-equivariant } & \iff & t_1 = 0, \\
\widetilde{K}(t_1,t_2) \textrm{ is $\KK^*$-equivariant } & \iff & t_1 = 0,
\end{eqnarray*}
and we have the equality of two-periodic complexes
\begin{equation}\label{augmentedcomplex}
\widetilde{K}(t_1,0) = K(t_1) \otimes \Lambda^\bullet \(p^*E\).
\end{equation}
Furthermore, we assume
\begin{eqnarray*}
K(t_1) \textrm{ is strictly exact off $Y$ } & \iff & t_1 \neq 0, \\
\widetilde{K}(t_1,t_2) \textrm{ is strictly exact off $Y$ } & \iff & (t_1,t_2) \neq 0.
\end{eqnarray*}
We recall from \cite{PV algebraic construction} that a two-periodic complex is strictly exact off $Y$ if it is exact off $Y$ and the images of the maps are subbundles.

As a consequence, we get localized Chern characters
\begin{eqnarray*}
\forall ~ t_1 \neq 0 ~,~~ \Ch^X_Y(K(t_1)) & \in & A(Y \to X)_\QQ, \\
\forall ~ (t_1,t_2) \neq 0 ~,~~ \Ch^X_Y(\widetilde{K}(t_1,t_2)) & \in & A(Y \to X)_\QQ
\end{eqnarray*}
in the bivariant Chow rings.
For their constructions, we refer to \cite{Fulton} for complexes and to \cite{PV algebraic construction} for two-periodic complexes.

\begin{dfn}
We call virtual class for the above input data the Chow class\footnote{In this formula, we recall that the projection $p \circ j \colon Y \to S$ is proper.
}
\begin{equation}\label{Chowclass}
c_\vir := (p \circ j)_*\left(\Td(T) \cup \Ch^X_Y(K(1))[X] \right) \in A_*(S).
\end{equation}
We call Hodge virtual class the following Chow class
\begin{equation}\label{def Hodge integral}
c_\mathrm{top}(E) \cdot c_\vir \in A_*(S),
\end{equation}
where $c_\mathrm{top}$ stands for the top Chern class of a vector bundle.
Hodge integrals refer to integrals of the Hodge virtual class against other Chow classes on the space $S$.
\end{dfn}

\begin{rem}
The reason for the name ``virtual class'' is explained in Section \ref{section2}.
Indeed, once we choose appropriate input data, we prove it corresponds to the virtual class from the quantum singularity theory or even from Gromov--Witten theory of hypersurfaces or more general GLSM, see \cite[Section 6]{Guere8}.
Moreover, the vector bundle $E$ plays the role of the Hodge bundle over the moduli space of stable curves, hence the second definition.
\end{rem}

\begin{rem}
It is a huge open challenge to express the virtual class $c_\vir$ in a simple way. However, the Hodge virtual class can be expressed in terms of the fixed loci of the $\KK^*$-action, see Theorem \ref{loc Hodge}, and thus it is often computable.
We also recall that Hodge integrals play an important role in the definition of the Double Ramification hierarchy, see \cite{Buryak DR}.
\end{rem}

\begin{pro}\label{eq Hodge}
We have an equality
$$c_\mathrm{top}(E) \cdot c_\vir = (-1)^{\rk(E)} (p \circ j)_*\left(\Td(T \oplus j^*p^*E^\vee) \cup \Ch^X_Y(\widetilde{K}(0,1))[X] \right)$$
in the Chow ring of $S$, where $\rk$ stands for the rank of a vector bundle.
\end{pro}

\begin{proof}
Using invariance of the localized Chern character by homotopy and equation \eqref{augmentedcomplex}, we see that
\begin{eqnarray*}
\forall t_1,t_2 \neq 0 ~,~~ \Ch^X_Y(\widetilde{K}(0,1)) & = & \Ch^X_Y(\widetilde{K}(t_1,t_2)) \\
& = & \Ch^X_Y(\widetilde{K}(t_1,0)) \\
& = & \Ch(\lambda_{-1} (p^*E)) \cdot \Ch^X_Y(K(t_1)) \\
& = & \Ch(\lambda_{-1} (p^*E)) \cdot \Ch^X_Y(K(1)),
\end{eqnarray*}
where the lambda-class is defined in K-theory on a vector bundle $W$ by
$$\lambda_t (W) = \sum_{q \geq 0} \Lambda^q W ~ t^q \in K^0[t].$$
Therefore, we obtain the desired equality after using the classical formulae \newline $\Ch(\lambda_{-1}(W^\vee)) \Td(W) = c_\mathrm{top}(W)$ and $c_\mathrm{top}(W^\vee) = (-1)^{\rk(W)} c_\mathrm{top}(W).$
\end{proof}

\subsection{Localization formula}
We recall all input data, except the global section $\alpha'\in H^0(X,V^\vee)$, are $\KK^*$-equivariant.
We denote by $q$ the equivariant parameter, by $A_*^{\KK^*}(\cdot)$ the equivariant Chow ring, and by $A_*^{\KK^*}(\cdot)_q$ the ring obtained from it by inverting the equivariant parameter $q$, see \cite{Edidin} for a detailed construction of the equivariant Chow ring.

\begin{pro}[{\cite[Thm 1]{Edidin2}}]\label{isom}
We have an isomorphism of groups
$$A_*^{\KK^*}(Y_F)_q \simeq A_*^{\KK^*}(Y)_q \simeq A_*^{\KK^*}(X)_q,$$
given by the equivariant pushforward of embeddings.
\end{pro}

Since the substack $Y$ is a complete intersection inside the stack $X$, we have an explicit description of the second isomorphism above, yielding a localization formula as the one proved by Atiyah--Bott \cite{Atiyah} in equivariant cohomology.

\begin{thm}[Localization formula]
Denote by $N_j=N_{Y\subset X}$ the normal bundle of the local complete intersection, it is a $\KK^*$-equivariant vector bundle over $Y$.
We have the following formula
$$[X] = j_*\(\cfrac{[Y]}{c^{\KK^*}_\mathrm{top}(N_j)}\) \in A^{\KK^*}_*(X)_q,$$
where $c^{\KK^*}_\mathrm{top}$ is the $\KK^*$-equivariant\footnote{We put the upper-script $\KK^*$ to emphasize that it is the equivariant top Chern class, even though it should be clear from the context. In particular, the localized Chern characters in the proof below are also $\KK^*$-equivariant, although we do not write the upper-script $\KK^*$.} top Chern class.
Furthermore, if the closed immersion $\iota_F \colon Y_F \hookrightarrow Y$ of the fixed locus inside $Y$ is also a local complete intersection, we have the same result replacing $Y$ by $Y_F$, $j$ by $j \circ \iota_F$, and the normal bundle $N_j$ by $N_{j \circ \iota_F}$.
\end{thm}

\begin{proof}
By the surjectivity of the second map in Proposition \ref{isom}, there exists an equivariant class $a \in A_*^{\KK^*}(Y)_q$ such that
$$[X] = j_*(a) \in A_*^{\KK^*}(X)_q.$$
Therefore, we have
$$[Y] = j^*[X] = j^*j_*(a) = c^{\KK^*}_\mathrm{top}(N_j) \cup a,$$
where in the last equality we use that $j$ is a local complete intersection morphism.
Therefore, dividing\footnote{Since the substack $Y \subset X$ is $\KK^*$-invariant, its normal bundle has no fixed part under $\KK^*$ and its equivariant top Chern class is then invertible.} both sides by $c^{\KK^*}_\mathrm{top}(N_j)$, we get the result.
\end{proof}

Applying the localization formula to the right-hand side of the equality in Proposition \ref{eq Hodge}, we obtain a simple formula for the Hodge virtual class.

\begin{thm}\label{loc Hodge}
The Hodge virtual class equals
$$c_\mathrm{top}(E^\vee) \cdot c_\vir = \lim_{q\to 0} c^{\KK^*}_\mathrm{top}(E^\vee) \cdot (p \circ j)_*\(\cfrac{c^{\KK^*}_\mathrm{top}(j^*V)}{c^{\KK^*}_\mathrm{top}(N_j)} \cdot \cfrac{\Td(j^*T)}{\Td(j^*V)}\),$$
where we recall the map $p \circ j$ is proper.
On the right-hand side, the Todd, Chern, and top Chern classes are all taken $\KK^*$-equivariantly.
\end{thm}

\begin{rem}
The Chow class on the right-hand side lives in the equivariant Chow ring $A^{\KK^*}_*(S)_q \simeq A_*(S)(\!(q)\!)$, as the action on $S$ is trivial.
Precisely, Theorem \ref{loc Hodge} means that it contains no negative powers of $q$ and that the constant term in $q$, that we interpret as a limit $q\to 0$, equals the left-hand side.
\end{rem}

\begin{proof}
By Proposition \ref{eq Hodge}, we rewrite the Hodge virtual class in terms of the localized Chern character of the two-periodic complex $\widetilde{K}(0,1)$. Since it is $\KK^*$-equivariant, we use the localization formula to get
\begin{eqnarray*}
\Ch^X_Y(\widetilde{K}(0,1))[X] & = & \Ch^X_Y(\widetilde{K}(0,1)) \left[ j_*\(\cfrac{[Y]}{c^{\KK^*}_\mathrm{top}(N_j)}\)\right] \\
& = & \Ch^{Y}_{Y}(j^*\widetilde{K}(0,1)) \left[ \cfrac{[Y]}{c^{\KK^*}_\mathrm{top}(N_j)}\right] \\
& = & \cfrac{\Ch(j^*\widetilde{K}(0,1))}{c^{\KK^*}_\mathrm{top}(N_j)} \in A^{\KK^*}_*(Y)_q. \\
\end{eqnarray*}
Furthermore, we have a K-theoretic equality
$$j^*\widetilde{K}(0,1) = \lambda_{-1} (j^*V^\vee \oplus j^* p^* E) = \lambda_{-1} (j^*V^\vee) \otimes \lambda_{-1}(j^*p^* E).$$
Thus, we obtain
$$\Td(T \oplus j^*p^*E^\vee) \cdot \Ch^X_Y(\widetilde{K}(0,1))[X] = j^*p^*(c^{\KK^*}_\mathrm{top}(E^\vee)) \cdot \cfrac{\Td(j^*T)}{\Td(j^*V)} \cdot \cfrac{c^{\KK^*}_\mathrm{top}(j^*V)}{c^{\KK^*}_\mathrm{top}(N_j)}$$
in the ring $A^{\KK^*}_*(Y)_q$.
Eventually, we push-forward to the space $S$ and get
\begin{equation}\label{eq1}
(p \circ j)_*\left(\Td(T \oplus E^\vee) \cup \Ch^X_Y(\widetilde{K}(0,1))[X] \right) = c^{\KK^*}_\mathrm{top}(E^\vee) \cdot (p \circ j)_*\(\cfrac{\Td(j^*T)}{\Td(j^*V)} \cdot \cfrac{c^{\KK^*}_\mathrm{top}(j^*V)}{c^{\KK^*}_\mathrm{top}(N_j)}\).
\end{equation}
Notice that at this point, it is a $\KK^*$-equivariant equality, taking place in the ring $A^{\KK^*}_*(S)_q$, which equals the ring $A_*(S)(\!(q)\!)$ since the $\KK^*$-action on $S$ is trivial.
It is not clear that the right-hand side of equality \eqref{eq1} contains no negative powers of the equivariant parameter $q$, but it follows from  the fact the left-hand side has a non-equivariant limit $q \to 0$.

At last, once we take the non-equivariant limit, or equivalently the constant term in $q$, in equality \eqref{eq1}, the left-hand side becomes the Hodge virtual class by Proposition \ref{eq Hodge} and we obtain the desired equality.
\end{proof}

\section{Application to FJRW theory}\label{section2}
In this section, we work over the base field $\KK=\CC$ and we give an application of Theorem \ref{loc Hodge} to the quantum singularity (FJRW) theory, shedding new light on the results of \cite{Guere1,Guere2}.
We then use most of the same notations and refer to \cite{Guere1,Guere2} for more details.

Once for all, we fix a quasi-homogeneous polynomial $W$ of chain type
$$W(x_1,\dotsc,x_N) = x_1^{a_1}x_2 + \dotsb + x_{N-1}^{a_{N-1}}x_N+x_N^{a_N},$$
where integers $a_1, \dotsc, a_N$ are positive.
We denote by $d$ its degree and by $w_1, \dotsc, w_N$ the weights of the variables $x_1, \dotsc, x_N$.
We also fix two non-negative integers $g$ and $n$ such that $2g-2+n>0$, i.e.~the space of stable curves $\overline{\cM}_{g,n}$ is non-empty.
We further consider an admissible group $G$, which is a subgroup of the maximal group $\mathrm{Aut}(W)$ of (diagonal) symmetries of $W$, containing the grading element $\grj$, see \cite[Equation (3)]{Guere1}.

\subsection{Hodge virtual cycle map}
Let us consider the Landau--Ginzburg orbifold $(W,G)$ (see \cite[Definition 1.2]{Guere1}) and denote by $\st$ the state space of its FJRW theory.
It decomposes as
$$\st = \bigoplus_{\gamma \in G} \st_\gamma$$
where we recall that $\gamma$ is a diagonal matrix $\gamma = \mathrm{diag}(\gamma_1, \dotsc, \gamma_N)$ and that $\gamma$ is called broad if at least one of its entries equals $1$ and narrow otherwise.
In particular, for $\gamma$ narrow, we have $\st_\gamma \simeq \CC$.

For any $\overline{\gamma} = (\gamma(1), \dotsc,\gamma(n)) \in G^n$ satisfying the selection rule
$$\gamma(1) \cdot \dotsm \cdot \gamma(n) = \grj^{2g-2+n},$$
where $\grj \in G$ is the grading element, we have a moduli space of $(W,G)$-spin curves, that we denote by $\sS_{g,n}(\overline{\gamma})$, see \cite[Section 2]{FJRW1} or \cite[Section 1.2]{Guere2}.
We denote by $\sS_{g,n}$ the disjoint union over all possible $\overline{\gamma}$ and we have a finite map $o \colon \sS_{g,n} \to \overline{\cM}_{g,n}$.

The virtual cycle map\footnote{We use the construction of the virtual cycle map by Polishchuk and Vaintrob \cite{Polish1}.} is a linear map
$$c_\vir \colon \st^{\otimes n} \to A^*(\sS_{g,n})$$
such that for entries $u_1 \in \st_{\gamma(1)}, \dotsc, u_n \in \st_{\gamma(n)}$ we have
$$c_\vir(u_1, \dotsc, u_n) \in H^{2 \mathrm{degvir}}(\sS_{g,n}(\gamma(1), \dotsc, \gamma(n)),$$
where the integer $\mathrm{degvir}$ is explicitly determined by the genus $g$ and the matrices $\gamma(1), \dotsc, \gamma(n)$, see \cite[Equation (16)]{Guere1}.

The Hodge bundle $\mathbb{E}$ on the moduli space $\overline{\cM}_{g,n}$ is the rank-$g$ vector bundle given by $\pi_* \omega$, where $\pi$ is the map from the universal curve and $\omega$ is the relative dualizing sheaf. Its fiber on the point representing a curve $C$ is then $H^0(C,\omega_C)$.

\begin{dfn}
We call Hodge virtual cycle map the product
$$\lambda_g \cdot c_\vir \colon \st^{\otimes n} \to A^*(\sS_{g,n}),$$
where $\lambda_g := c_\mathrm{top}(\mathbb{E})$ is the top Chern class of the (pull-back of the) Hodge bundle. 
\end{dfn}

From now on, we fix entries $u_1 \in \st_{\gamma(1)}, \dotsc, u_n \in \st_{\gamma(n)}$ which are invariant under the maximal group $\mathrm{Aut}(W)$, see \cite[Section 1.3]{Guere2} for an explicit description.
In particular, we are given a subset $\rR_{\gamma(i)} \subset \left\lbrace x_1, \dotsc, x_N \right\rbrace$ of variables for each marking, see \cite[Definition 1.5]{Guere1}.


Before stating Theorem \ref{mainFJRW}, we introduce a few more notations.
Let us denote by $\pi \colon \cC \to \sS_{g,n}$ the universal curve over the moduli space of $(W,G)$-spin curves and by $\cL_1, \dotsc, \cL_N$ the universal line bundles.
In \cite[Equation (72)]{Guere1} or \cite[Equation (11)]{Guere2}, we define the modified line bundles $\cL^\rR_1, \dotsc, \cL^\rR_N$ on $\cC$, which are obtained from the universal line bundles by twisting down some of the markings.
We also need the following definition.

\begin{dfn}\label{equivEuler}
The equivariant Euler class of a vector bundle $V$ on a space $S$ is defined by
\begin{equation}\label{totChern}
e_q(V) = q^{\rk(V)} \cdot \(1+\cfrac{c_1(V)}{q}+\cfrac{c_2(V)}{q^2}+\dotsb+\cfrac{c_{\rk(V)}(V)}{q^{\rk(V)}}\),
\end{equation}
and extended multiplicatively to K-theory as a map $e_q \colon K^0(S) \to A_*(S)[ q,q^{-1} ]\!]$.
\end{dfn}

\begin{rem}
Let $\CC^*$ act trivially on a space $S$ and fiberwise on a vector bundle $V$. The $\CC^*$-equivariant top Chern class of $V$ equals the equivariant Euler class, i.e.~$c^{\CC^*}_\mathrm{top}(V)=e_q(V)$, where $q$ is the equivariant parameter.
\end{rem}

\begin{thm}\label{mainFJRW}
Under the previous assumptions and notations, we have
\begin{equation}\label{formFJRW}
\lambda_g \cdot c_\vir(u_1, \dotsc, u_n) = \lim_{q\to 0} e_{-q_{N+1}}(\mathbb{E}) \cdot \prod_{j=1}^N e_{q_j}(-R^\bullet\pi_*(\cL^\rR_j)),
\end{equation}
where $q_1:=q$ and $q_{j+1}:=(-a_1) \dotsm (-a_j) ~ q$ for $1 \leq j \leq N$.
\end{thm}

\begin{proof}
First of all, we take resolutions of the higher push-forwards $R^\bullet\pi_*(\cL^\rR_j)$ for all $j$ by vector bundles over $\sS_{g,n}$
$$R\pi_*\cL^\rR_j = [A_j \to \widetilde{B}_j],$$
such that there exist the appropriate morphisms from \cite[Equation (12)]{Guere2}
\begin{equation}\label{widealpha}
\begin{array}{lcl}
\widetilde{\alpha}_j & \colon & \cO \rightarrow \cS^{a_{j-1}} A_{j-1}^\vee \otimes \widetilde{B}_j^\vee \oplus (\cS^{a_j-1} A_j^\vee \otimes A_{j+1}^\vee) \otimes \widetilde{B}_j^\vee, \\
\widetilde{\beta}_j & \colon & \widetilde{B}_j^\vee \rightarrow A_j^\vee,
\end{array}
\end{equation}
with the convention $(A_0,A_{N+1})=(0,A_N)$.

Moreover, we decompose the morphism $\widetilde{\alpha}_N$ into the sum $\widetilde{\alpha}'_N+\widetilde{\alpha}''_N$ where
$$\widetilde{\alpha}'_N \colon \cO \rightarrow \cS^{a_{N-1}} A_{N-1}^\vee \otimes \widetilde{B}_N^\vee$$
and we also consider the morphism from \cite[Equation (19)]{Guere2}
$$\widetilde{\alpha}_{N+1} \colon \cO \rightarrow \cS^{a_N} A_N^\vee \otimes \mathbb{E}.$$

We apply Theorem \ref{loc Hodge} to the following data:
\begin{itemize}
\item $S := \sS_{g,n}(\gamma(1), \dotsc, \gamma(n))$ is the moduli space of $(W,G)$-spin curves,
\item $E := \mathbb{E}$ is the Hodge bundle,
\item $X := \mathrm{Tot}\(A_1 \oplus \dotsb \oplus A_N\)$ is the total space of the vector bundle, with $p \colon X \to S$ the projection,
\item $Y := \sS_{g,n}(\gamma(1), \dotsc, \gamma(n))$ embedded in $X$ via the zero section $Y \hookrightarrow X$,
\item $V := p^*B_1 \oplus \dotsb \oplus p^*B_N$,
\item $T := B_1 \oplus \dotsb \oplus B_N$,
\item $\alpha := \widetilde{\alpha}_1 + \dotsb + \widetilde{\alpha}_{N-1} + \widetilde{\alpha}'_N$ viewed as a global section of $V^\vee$ on $X$,
\item $\alpha' := \widetilde{\alpha}''_N$ viewed as a global section of $V^\vee$ on $X$,
\item $\alpha'' := \widetilde{\alpha}_{N+1}$ viewed as a global section of $p^*E$ on $X$,
\item $\beta := \widetilde{\beta}_1 + \dotsb + \widetilde{\beta}_N$ viewed as a global section of $V$ on $X$.
\end{itemize}
It follows from \cite[Section 3.5]{Guere1} that $\beta(\alpha)=\beta(\alpha')=0$.
Furthermore, we take the following $\CC^*$-action:
\begin{itemize}
\item trivial on $S$,
\item scaling fibers with weight $1$ on $A_1$ and on $B_1$, i.e.~$\lambda \cdot v = \lambda ~ v$ on a vector $v$,
\item scaling fibers with weight $(-a_1) \dotsm (-a_j)$ on $A_{j+1}$ and on $B_{j+1}$, i.e.~$\lambda \cdot v = \lambda^{(-a_1) \dotsm (-a_j)} ~ v$ on a vector $v$,
\item scaling fibers with weight $(-a_1) \dotsm (-a_{N})$ on the Hodge bundle $\mathbb{E}$.
\end{itemize}
It is straightforward to see that the global sections $\alpha, \alpha''$, and $\beta$ are $\CC^*$-equivariant, and that $\alpha'$ is not.
Moreover, the $\CC^*$-fixed locus in $X$ is given by the constraint
$$\forall \lambda \in \CC^*~,~~\lambda \cdot (x_1, \dotsc, x_N) = (\lambda ~ x_1,\dotsc, \lambda^{(-a_1) \dotsm (-a_{N-1})} ~ x_N) = (x_1, \dotsc, x_N),$$
yielding $(x_1, \dotsc, x_N)=0$, i.e.~the $\CC^*$-fixed locus is $Y_F=Y=S$.

Following Section \ref{section1}, we form the two-periodic complexes $K(t_1)$ and $K(t_1,t_2)$ for $(t_1,t_2) \in \CC^2$, and they are $\CC^*$-equivariant when $t_1=0$.
Looking at the common vanishing locus of the global sections $\alpha, \alpha''$, and $\beta$, it follows that $K(t_1)$ is strictly exact when $t_1 \neq 0$ and $K(t_1,t_2)$ is strictly exact when $(t_1,t_2) \neq 0$.
As a consequence, all assumptions from Section \ref{section1} are fulfilled.

By definition of the virtual cycle map, the equality
\begin{equation}\label{Chiodo}
c_\vir(u_1, \dotsc, u_n) = p_*\left(\Td(T) \cup \Ch^X_Y(K(1))[X] \right) \in A_*(S)
\end{equation}
is exactly \cite[Lemma 5.3.8]{ChiodoJAG}.
Therefore, we obtain by Theorem \ref{loc Hodge}
\begin{eqnarray*}
c_\mathrm{top}(E^\vee) \cdot c_\vir(u_1, \dotsc, u_n) & = & \lim_{q\to 0} c^{\KK^*}_\mathrm{top}(E^\vee) \cdot (p \circ j)_*\(\cfrac{c^{\KK^*}_\mathrm{top}(j^*V)}{c^{\KK^*}_\mathrm{top}(N_j)} \cdot \cfrac{\Td(j^*T)}{\Td(j^*V)}\) \\
& = & \lim_{q\to 0} c^{\KK^*}_\mathrm{top}(E^\vee) \cdot \cfrac{c^{\KK^*}_\mathrm{top}(j^*V)}{c^{\KK^*}_\mathrm{top}(N_j)} \\
& = & \lim_{q\to 0} e_{q_{N+1}}(\mathbb{E}^\vee) \cdot \cfrac{e_{q_1}(B_1) \dotsm e_{q_N}(B_N)}{e_{q_1}(A_1) \dotsm e_{q_N}(A_N)},
\end{eqnarray*}
where we use that the normal bundle $N_j$ equals the vector bundle $A_1 \oplus \dotsb \oplus A_N$.
At last, the equality $e_q(E^\vee) = (-1)^{\rk E} e_{-q}(E)$ concludes the proof.
\end{proof}

\subsection{Comparaison with previous work}
Theorem \ref{mainFJRW} computes exactly the same class as \cite[Theorem 2.2]{Guere2}. We then have to compare the two formulae.

\begin{rem}
It is interesting to see that Equation \eqref{formFJRW} from Theorem \ref{loc Hodge} was already written in \cite[Equation (24)]{Guere2}, where it was deduced from the computation of the sum over dual graphs, see \cite[Section 3]{Guere2}.
In particular, the expression of Equation \eqref{formFJRW} as a sum over dual graphs is presented in \cite[Corollary 3.5]{Guere2}.
Nevertheless, we give below a more comprehensive way to understand the relationship between Theorem \ref{mainFJRW} and \cite[Theorem 2.2]{Guere2}.	
\end{rem}

In \cite{Guere1,Guere2}, we use the notion of recursive complexes to obtain the formula for the Hodge virtual cycle map and it uses a multiplicative characteristic class defined as follows.
On a vector bundle $V$ over a space $S$, it is
$$\fc_t(V) = \Ch(\lambda_{-t}V^\vee) \Td(V) \in A_*(S)[t]$$
and in terms of its roots $\alpha_1,\dotsc,\alpha_v$, it is
\begin{equation*}
\fc_t(V) = \prod_{k=1}^v \frac{e^{\alpha_k}-t}{e^{\alpha_k}-1} \cdot \alpha_k.
\end{equation*}
It is multiplicative on vector bundles and then extended multiplicatively to K-theory into a function $\fc_t \colon K^0(S) \rightarrow A_*(S)[\![ t ]\!]$.
We have the fundamental property
\begin{equation}\label{multiplicativity}
\forall R,R' \in K^0(S) ~,~~ \fc_t(R+R') = \fc_t(R) \cdot \fc_t(R').
\end{equation}

Note also that the characteristic class $\fc_t$ is actually defined for $t \neq 1$, and not only for a formal parameter $t$, using Chern characters. Explicitly, for $R \in K^0(S)$, we have
\begin{equation}\label{explfc}
\fc_t(R) = (1-t)^{\Ch_0(R)} \cdot \exp \(-\sum_{l \geq 1} s_l(t) \Ch_l(R)\),
\end{equation}
with the functions
\begin{equation}\label{parametresl}
s_l(t) = \cfrac{B_l(0)}{l} + (-1)^l \sum\limits_{k=1}^l (k-1)! \left( \frac{t}{1-t} \right)^k \gamma(l,k). 
\end{equation}
Here, the number $\gamma(l,k)$ is defined by the generating function
\begin{equation*}
\sum_{l \geq 0} \gamma(l,k) \frac{z^l}{l!} := \frac{(e^z-1)^k}{k!}.
\end{equation*}
We notice that $\gamma(l,k)$ vanishes for $k>l$ and that the sum over $l$ in \eqref{explfc} is finite because $\Ch_l$ vanishes for $l > \dim(S)$.

Similarly, the equivariant Euler class (see Definition \ref{equivEuler}) is actually defined for $q \neq 0$ using Chern characters.
Explicitly,  for $R \in K^0(S)$, we have
\begin{equation}\label{explc}
e_q(R) = q^{\Ch_0(R)} \cdot \exp \(-\sum_{l \geq 1} \cfrac{(l-1)!}{(-q)^l} \Ch_l(R)\).
\end{equation}
It also has the multiplicativity property
\begin{equation}\label{multiplicativity2}
\forall R,R' \in K^0(S) ~,~~ e_q(R+R') = e_q(R) \cdot e_q(R'),
\end{equation}
and in terms of roots $\alpha_1,\dotsc,\alpha_v$ of a vector bundle $V$, it takes the simple form
$$e_q(V) = \prod_{k=1}^v q+\alpha_k.$$

\begin{pro}
Let $R \in K^0(S)$ and $q$ be a formal parameter or be in $\KK^*$.
We have the relation
\begin{equation}\label{compfce}
e_q(R) = \fc_{e^{-q}}(R) \cdot \cfrac{\Td(R \otimes \cO(q))}{\Td(R)},
\end{equation}
where $\cO(q)$ is a formal line bundle with first Chern class $q$.
Precisely, we have
\begin{equation}\label{Tdform}
\cfrac{\Td(R \otimes \cO(q))}{\Td(R)} = \exp \(-\sum_{\substack{k \geq 0 \\ l>k}} \cfrac{B_l(0)}{l} q^{l-k} \Ch_k(R)\).
\end{equation}
\end{pro}

\begin{proof}
The equality
$$\Td(R) = \exp \(-\sum_{l>1} \cfrac{B_l(0)}{l} \Ch_l(R)\)$$
is easy to prove using the multiplicativity of the Todd class and the formula
$$\Td(L) = \cfrac{c_1(L)}{1-e^{-c_1(L)}}$$
for a line bundle $L$.
Thus, we get Equation \eqref{Tdform} using
$$\Ch_l(R \otimes \cO(q)) = \sum_{k=0}^l \Ch_k(R) \cfrac{q^{l-k}}{(l-k)!}.$$
Similarly, since the classes $e_q$, $\fc_{e^{-q}}$, and $\Td$ are multiplicative, it is enough to check equation \eqref{compfce} on a line bundle $L$. Denoting $\alpha := c_1(L)$, we obtain
\begin{eqnarray*}
\fc_{e^{-q}}(L) \cdot \cfrac{\Td(L \otimes \cO(q))}{\Td(L)} & = & \cfrac{e^\alpha-e^{-q}}{e^\alpha-1} \cdot \alpha \cdot \cfrac{\alpha+q}{1-e^{-\alpha-q}} \cdot \cfrac{1-e^{-\alpha}}{\alpha} \\
& = & \alpha+q = e_q(L).
\end{eqnarray*}
\end{proof}

\begin{dfn}
We say that a formal power series in $q,q^{-1}$ is convergent when $q \to 0$ if all coefficients of negative powers in $q$ are zero.
These coefficients yield relations in the coefficient ring of the formal power series.
Moreover, we call limit of a convergent formal power series its constant term in $q$. 
\end{dfn}

\begin{cor}\label{tautrel}
Let $R_1, \dotsc,R_N \in K^0(S)$ and $k_1, \dotsc,k_N \in \ZZ$. We set $q_j:=k_j \cdot q$, $t_j:=t^{k_j}$, and $t:=e^{-q}$.
Then, the formal power series $\prod_{j=1}^N \fc_{t_j}(R_j)$ is convergent when $q \to 0$ if and only if the formal power series $\prod_{j=1}^N e_{q_j}(R_j)$ is convergent when $q \to 0$.
Furthermore, under convergence, we have
$$\lim_{q \to 0} \prod_{j=1}^N e_{q_j}(R_j) = \lim_{q \to 0} \prod_{j=1}^N \fc_{t_j}(R_j) = \lim_{t \to 1} \prod_{j=1}^N \fc_{t_j}(R_j)$$
and the two sets of relations in $A_*(S)$ are equivalent.
\end{cor}

\begin{proof}
Since we have
$$\exp \(-\sum_{\substack{k \geq 0 \\ l>k}} \cfrac{B_l(0)}{l} q^{l-k} \Ch_k(R)\) = 1+\frac{q}{2}+q^2 \cdot f(q) + q \cdot g(q),$$
with $f(q) \in \CC[\![q]\!]$ and $g(q) \in A^{\deg \geq 1}(S)[\![q]\!]$, then we get
$$\prod_{j=1}^N \cfrac{e_{q_j}(R_j)}{\fc_{t_j}(R_j)} = 1+ O(q)$$
and the claims follow easily.
\end{proof}

\begin{cor}\label{comp1}
Theorem \ref{mainFJRW} gives the same result as \cite[Theorem 2.2]{Guere2}.
In particular, it recovers \cite[Theorem 3.21]{Guere1} as a special case for genus zero.
Furthermore, the set of tautological relations presented in \cite[Section 2.3]{Guere2} is equivalent to the set of tautological relations obtained by looking at the coefficients of negative powers of $q$ in the right-hand side of Equation \eqref{formFJRW}.
\end{cor}


\subsection{Equivariant mirror symmetry}
\cite[Section 4]{Guere1} can be entirely rewritten with the specialization of the twisted theory given by
$$s_0^j := \frac{1}{q_j} \quad \textrm{and} \quad s_l^j := \cfrac{(l-1)!}{(-q_j)^l} ~~ \textrm{for $l \geq 1$,}$$
with $q_1 := q$ and $q_{j+1} := (-a_1) \dotsm (-a_j) ~ q$ for $1 \leq j < N$, instead of the specialization given by \cite[Equation (67)]{Guere1}.
According to Corollary \ref{comp1}, it recovers the same big I-function of \cite[Theorem 4.2]{Guere1} and the same small I-function and Picard--Fuchs equation of \cite[Theorem 4.4]{Guere1} once we take the limit $q \to 0$.
However, it is interesting to consider the equivariant version of these results, i.e.~without taking the limit $q \to 0$.

In \cite[Theorem 4.2]{Guere1}, the only change is that, with the notations from there, the contribution $M_j(\overline{\gamma})$ becomes
\begin{equation*}
M_j(\overline{\gamma}) = \left\lbrace
\begin{split}
\prod_{0 \leq m \leq \cD^\rR_j(\overline{\gamma}) - 1}  (\omega^\rR_j(\overline{\Gamma}) + m) z + q_j & \qquad \textrm{ when $\cD^\rR_j(\overline{\gamma}) \geq 1$}, \\
1 & \qquad \textrm{ when $\cD^\rR_j(\overline{\gamma}) = 0$}, \\
\prod_{1 \leq m \leq -\cD^\rR_j(\overline{\gamma})}  \cfrac{1}{(\omega^\rR_j(\overline{\Gamma}) - m) z + q_j}, & \qquad \textrm{ when $\cD^\rR_j(\overline{\gamma}) \leq -1$}. \\
\end{split}
\right.
\end{equation*}
As a consequence, we deduce the equivariant small I-funtion for chain polynomials of Calabi--Yau type\footnote{The Calabi--Yau condition means that the degree equals the sum of the weights, i.e.~$d=w_1+\dotsb+w_N$.}

\begin{thm}[Equivariant Mirror Symmetry]\label{equiMS}
Let $W$ be a chain polynomial of Calabi--Yau type.
The equivariant I-function\footnote{There is a typo in \cite[Equation (98)]{Guere1} as the condition $b \geq 0$ is not specified.} defined for $t \in \CC^*$ by
\begin{equation}\label{finalformulaI}
I(t,-z) = -z \sum_{k=1}^\infty t^k \frac{\prod_{j=1}^N \prod_{\substack{\delta_j < b < \fq_j k \\ b \geq 0, \langle b \rangle = \langle \fq_j k \rangle}} (bz+q_j)}{\prod_{0 < b < k} bz} e_{\grj^k}, \quad \delta_j := - \delta_{\{N-j \textrm{ is odd}\}}
\end{equation}
lies on the Lagrangian cone $\cL$ of the equivariant FJRW theory of the Landau--Ginzburg orbifold $(W,G)$. 
This function satisfies the Picard--Fuchs equation
\begin{equation}\label{PicardFuchsequation}
\biggl[ t^d \prod_{j=1}^N \prod_{c=0}^{w_j-1} (\fq_j z t \frac{\partial}{\partial t} + cz + q_j) - \prod_{c=1}^d (z t \frac{\partial}{\partial t} - c z)\biggr] \cdot I(t, -z) = 0.
\end{equation}
\end{thm}



\begin{bibdiv}
\begin{biblist}

\bib{Atiyah}{article}{
	author={Atiyah, Michael},
	author={Bott, Raoul},
	title={The moment map and equivariant cohomology},
	journal={Topology},
	volume={23},
	date={1984},
	number={},
	pages={1-28},
}


\bib{BuryakDR}{article}{
	author={Buryak, Alexandr},
	title={Double ramification cycles and integrable hierarchies},
	journal={Communications in Mathematical Physics},
	volume={336},
	date={2015},
	number={3},
	pages={1085-1107},
}

\bib{Guere3}{article}{
	author={Buryak, Alexandr},
	author={Gu\'er\'e, J\'er\'emy},
	title={Towards a description of the double ramification hierarchy for Witten's r-spin class},
	journal={Journal de Math\'ematiques Pures et Appliqu\'ees},
	volume={106},
	date={2016},
	number={5},
	pages={837-865},
}

\bib{ChiodoJAG}{article}{
	author={Chiodo, Alessandro},
	title={The Witten top Chern class via K-theory},
	journal={Journal of Algebraic Geometry},
	volume={15},
	date={2006},
	number={},
	pages={681-707},
}



\bib{LGCY}{article}{
	author={Chiodo, Alessandro},
	author={Iritani, Hiroshi},
	author={Ruan, Yongbin},
	title={Landau--Ginzburg/Calabi--Yau correspondence, global mirror symmetry and Orlov equivalence},
	journal={Publications math\'ematiques de l'IH\'ES},
	volume={119},
	date={2014},
	number={1},
	pages={127-216},
}

\bib{Guere8}{article}{
	author={Ciocan-Fontanine, Ionut},
	author={Favero, David},
	author={Gu\'er\'e, J\'er\'emy},
	author={Kim, Bumsig},
	author={Shoemaker, Mark},
	title={Fundamental Factorization of a GLSM, Part I: Construction},
	journal={available at arXiv:1802.05247},
	volume={},
	date={},
	number={},
	pages={},
}


\bib{Edidin}{article}{
	author={Edidin, Dan},
	author={Graham, William},
	title={Equivariant intersection theory},
	journal={Inventiones Math.},
	volume={131},
	date={1998},
	number={},
	pages={595-634},
}

\bib{Edidin2}{article}{
	author={Edidin, Dan},
	author={Graham, William},
	title={Localization in equivariant intersection theory and the Bott residue formula},
	journal={American Journal of Mathematics},
	volume={120},
	date={1998},
	number={3},
	pages={619-636},
}



\bib{FJRW1}{article}{
	author={Fan, Huijun},
	author={Jarvis, Tyler},
	author={Ruan, Yongbin},
	title={The Witten equation, mirror symmetry and quantum singularity theory},
	journal={Ann. of Math.},
	volume={178},
	date={2013},
	number={1},
	pages={1-106},
}   

\bib{FJRW2}{article}{
	author={Fan, Huijun},
	author={Jarvis, Tyler},
	author={Ruan, Yongbin},
	title={The Witten equation and its virtual fundamental cycle},
	journal={available at arXiv:0712.4025},
	volume={},
	date={},
	number={},
	pages={},
} 

\bib{GLSM}{article}{
	author={Fan, Huijun},
	author={Jarvis, Tyler},
	author={Ruan, Yongbin},
	title={A mathematical theory of the gauged linear sigma model},
	journal={Geometry and Topology},
	volume={},
	date={2017},
	number={},
	pages={},
} 

\bib{Fulton}{article}{
        author={Fulton, William},
        title={Intersection theory},
        journal={Springer--Verlag},
        volume={},
        date={},
        number={},
        pages={},
      }

\bib{Guere1}{article}{
	author={Gu\'er\'e, J\'er\'emy},
	title={A Landau--Ginzburg Mirror Theorem without Concavity},
	journal={Duke Mathematical Journal},
	volume={165},
	date={2016},
	number={13},
	pages={2461-2527},
}

\bib{Guere2}{article}{
	author={Gu\'er\'e, J\'er\'emy},
	title={Hodge integrals in FJRW theory},
	journal={Michigan Mathematical Journal},
	volume={66},
	date={2017},
	number={4},
	pages={831-854},
}

\bib{Guere10}{article}{
	author={Gu\'er\'e, J\'er\'emy},
	title={Hodge--Gromov--Witten theory},
	journal={},
	volume={},
	date={2019},
	number={},
	pages={},
}

\bib{Jeongseok}{article}{
	author={Kim, Bumsig},
	author={Oh, Jeongseok},
	title={Localized Chern characters for $2$-periodic complexes},
	journal={avalaible at arXiv:1804.03774},
	volume={},
	date={},
	number={},
	pages={},
}

\bib{PV algebraic construction}{article}{
   author={Polishchuk, Alexander},
   author={Vaintrob, Arkady},
   title={Algebraic construction of Witten's top Chern class},
   journal={Contemp. Math.},
   volume={276},
   date={2001},
   number={},
   pages={},
 }    

\bib{Polish1}{article}{
	author={Polishchuk, Alexander},
	author={Vaintrob, Arkady},
	title={Matrix factorizations and cohomological field theories},
	journal={J. Reine Angew. Math.},
	volume={714},
	date={2016},
	number={},
	pages={1-122},
} 





\end{biblist}
\end{bibdiv}
\end{document}